\documentclass[12pt,a4paper,reqno]{amsart}
\allowdisplaybreaks
\usepackage{amsmath}
\usepackage{amsfonts}
\usepackage{amssymb,amsthm,amsfonts,amsthm,latexsym,enumerate,url,cases}
\numberwithin{equation}{section}
\usepackage{mathrsfs}
\usepackage{hyperref}
\hypersetup{colorlinks=true,citecolor=blue,linkcolor=blue,urlcolor=blue}
\newtheorem{theorem*}{Theorem}
\newtheorem{lemma*}{Lemma}
\theoremstyle{plain}
\newtheorem{theorem}{Theorem}

\newtheorem{problem}{Problem}
\theoremstyle{definition}


\begin{document}

\title
[{Solutions to a Romanoff type problem}] {Solutions to a Romanoff type problem}

\author
[Y. Ding] {Yuchen Ding}

\address{(Yuchen Ding) School of Mathematical Science,  Yangzhou University, Yangzhou 225002, People's Republic of China}
\email{ycding@yzu.edu.cn}

\keywords{Romanoff theorem; Primes; Prime number theorem; Chebyshev estimate; Mertens estimate}
\subjclass[2010]{Primary 11P32, 11A41; Secondary 11B13.}

\begin{abstract}
We answer negatively a 2014 problem of Yang and Chen on Romanoff type representations. Sharp results involving their problem were also obtained in this article.
\end{abstract}
\maketitle

\baselineskip 18pt

\section{Introduction}
In 1849, de Polignac \cite{de1} made the conjecture that any odd number greater than 3 is the sum of a prime and a power of 2. But soon, he \cite{de2} recognized that 127 and 959 are two counterexamples. Actually, as de Polignac mentioned, these two  counterexamples were already pointed out in a 1752 letter from Euler to Goldbach. Along the positive line, Romanoff \cite{Ro} proved that there is a positive proportion of the odd numbers which can be represented by the sum of a prime and a power of 2. To answer a question of Romanoff, in 1950 Erd\H{o}s \cite{Er} constructed an arithmetic progression, none of whose member can be written as the sum of a prime and a power of 2. This immediately leads to a theorem of
van der Corput \cite{va}, which states that the odd numbers not with the form $p+2^m$, where $p\in\mathcal{P}$ and $m\in \mathbb{N}$, also possess positive lower density.
During the past decades, a large number of variants of the Romanoff theorem were established. To name only a few of them, see e.g. \cite{Ch1,Ch2,Ch3,ES,Els,Pa3} and the references listed therein.

For any set $\mathcal{A}$, let $\mathcal{A}(x)=\left|\mathcal{A}\cap[1,x]\right|$. A subset $\mathcal{B}$ of $\mathbb{N}$ is said to satisfy $c$-condition if $\mathcal{B}(cx)\gg \mathcal{B}(x)$ for some positive constant $c<1$. In the article of Yang and Chen \cite{Ch}, the following sumset
$$\mathcal{S}=\left\{p+b:p\in \mathcal{P},b\in \mathcal{B}\right\}$$
is considered, where $\mathcal{B}$ is a subset of $\mathbb{N}$ with $c$-condition. They proved that
\begin{equation}\label{e1}
\frac{x}{\log x}\min\left\{\mathcal{B}(x),\frac{\log x}{\log\log x}\right\}\ll \mathcal{S}(x)\ll \frac{x}{\log x}\min\left\{\mathcal{B}(x),\log x\right\}.
\end{equation}
As an application, Chen and Yang showed that
$$\#\left\{n\le x:n=p+2^{a^2}+2^{b^2}, p\in\mathcal{P},a,b\in\mathbb{N}\right\}\gg \frac{x}{\log\log x}.$$
Besides the above result, they also constructed a subset $\mathcal{B}$ of $\mathbb{N}$ such that
\begin{equation}\label{e2}
\mathcal{B}(x)=\frac{1+o(1)}{m+1}\left(\frac{\log x}{\log\log x}\right)^{m+1}
\end{equation}
and
$$\mathcal{S}(x)\ll \frac{x}{\log\log x},$$
where $m$ is an arbitrarily given integer. Yang and Chen \cite{Ch} then posed the following two problems for further research.

\begin{problem}\label{p1}
Does there exist a real number $\alpha>0$ and a subset $\mathcal{B}$ of $\mathbb{N}$ with $c$-condition such that $\mathcal{B}(x)\gg x^{\alpha}$ and $\mathcal{S}(x)\ll x/\log\log x$?
\end{problem}

\begin{problem}\label{p2}
Does there exist a positive integer $k$ such that the set of positive integers which can be represented as $p+\sum_{i=1}^{k}2^{m_i^2}$ with $p\in\mathcal{P}$ and $m_i\in \mathbb{N}$ has a positive lower density? If such $k$ exists, what is the minimal value of such $k$?
\end{problem}

Recently, the author \cite{Ding} gave a complete solution to {\bf Problem 2} by showing that $k=2$ is admissible. In this subsequent note, we turn to the investigations of {\bf Problem 1}. We give a negative answer to it via the following theorem.

\begin{theorem}\label{th1}
Let $\alpha>0$ be arbitrarily small given number and $\mathcal{B}$ a subset of $\mathbb{N}$ with $c$-condition such that $\mathcal{B}(x)\gg x^{\alpha}$. Then
$$
\mathcal{S}(x)=\#\left\{n\le x:n=p+b,p\in \mathcal{P},b\in \mathcal{B}\right\}\gg x/\log\log\log x.
$$
\end{theorem}

It would be of interest to show that the lower bound in Theorem \ref{th1} is in fact optimistic in some sense. Let's explain this by the following Theorem \ref{th2}.

From now on, let $\exp(y)$ denote $e^y$ for any real number $y$.

\begin{theorem}\label{th2} For any positive integer $m\ge2$, there exists a subset $\mathcal{B}$ of $\mathbb{N}$ with $c$-condition
such that
$$\mathcal{B}(x)\gg x\exp\left(-\frac{2}{m}(\log x)^{1/m}\log\log x\right)$$
and
$$\mathcal{S}(x)=\#\left\{p+b\le x:p\in \mathcal{P},b\in \mathcal{B}\right\}\ll x/\log\log\log x.$$
\end{theorem}

It is clear that
$$
x\exp\left(-\frac{2}{m}(\log x)^{1/m}\log\log x\right)> x^{1-\varepsilon}
$$
for any given $\varepsilon>0$, provided that $x$ is sufficiently large. Thus,
Theorem \ref{th2} surpasses our requirement significantly.

\section{Proof of Theorem \ref{th1}}
The standard method in the proof of the Romanoff type theorems lies in the investigations on moments of the representation function defined as
$$
f(n)=\#\left\{(p,b):n=p+b,p\in \mathcal{P},b\in \mathcal{B}\right\}.
$$
But, as it was already shown by Yang and Chen \cite{Ch}, this could only lead to the bound
$$
\mathcal{S}(x)\gg x/\log\log x.
$$
Our new idea is, instead of investigating $f(n)$, studying a certain truncated representation function
\begin{equation}\label{eq18-1}
f_\alpha(n)=\#\left\{(p,b):n=p+b,p\in \mathcal{P},b\in \mathcal{B},b<(\log n)^{1/\alpha}\right\}.
\end{equation}
We now proceed the details below.
\begin{proof}[Proof of Theorem \ref{th1}]
Since $\mathcal{B}$ is a subset of $\mathbb{N}$ with $c$-condition, there is a constant $c<1$ such that $\mathcal{B}(cx)\gg \mathcal{B}(x)$. For any given $\alpha>0$, there exists some $0<\delta<1$ (depending on $c$ and $\alpha$) so that
$$
\delta^{1/\alpha}>c.
$$

Let $f_\alpha(n)$ be defined as the one in Eq. (\ref{eq18-1}). Noting that
\begin{align*}
\mathcal{S}(x)&=\#\left\{n\le x: n=p+b,p\in \mathcal{P},b\in \mathcal{B}\right\}\\
&\ge \#\left\{n\le x: n=p+b,p\in \mathcal{P},b\in \mathcal{B},b<(\log x)^{1/\alpha}\right\}\\
&:=\mathcal{S}_\alpha(x), \quad \text{say},
\end{align*}
it suffices to prove $\mathcal{S}_\alpha(x)\gg x/\log\log\log x$. Employing the Cauchy--Schwarz inequality, we would obtain that
\begin{align}\label{eq18-2}
\Bigg(\sum_{n\le x}f_\alpha(n)\Bigg)^2\le \Bigg(\sum_{n\le x}f_\alpha^2(n)\Bigg)\mathcal{S}_\alpha(x)
\end{align}
since any number $n\le x$ with $f_\alpha(n)\ge 1$ will clearly be counted exactly once by $\mathcal{S}_\alpha(x)$. It can be seen for sufficiently large $x$ that
\begin{align*}
\sum_{n\le x}f_\alpha(n)\ge \sum_{x^\delta<n\le x}f_\alpha(n)\ge\sum_{\substack{x^\delta<p+b\le x\\ b<\left(\log x^\delta\right)^{1/\alpha}}}1\ge \sum_{x^{\delta}<p<x/2}1\sum_{b<(\log x^{\delta})^{1/\alpha}}1.
\end{align*}
Thus, we deduce from above and the prime number theorem that
\begin{align*}
\sum_{n\le x}f_\alpha(n)\gg_{c,\alpha} \frac{x}{\log x}\mathcal{B}\left((\log x^\delta)^{1/\alpha}\right)=\frac{x}{\log x}\mathcal{B}\left(\delta^{1/\alpha}(\log x)^{1/\alpha}\right),
\end{align*}
provided that $x$ is sufficiently large (in terms of $c$ and $\alpha$).
Recall that $\delta^{1/\alpha}>c$, we know
$$
\mathcal{B}\left(\delta^{1/\alpha}(\log x)^{1/\alpha}\right)\ge \mathcal{B}\left(c(\log x)^{1/\alpha}\right)\gg_{c,\alpha} \mathcal{B}\left((\log x)^{1/\alpha}\right),
$$
where the last inequality comes from the $c$-condition of $\mathcal{B}$.
It then follows that
\begin{align}\label{eq18-3}
\sum_{n\le x}f_\alpha(n)\gg_{c,\alpha} \frac{x}{\log x}\mathcal{B}\left((\log x)^{1/\alpha}\right).
\end{align}
We are now in a position to study the second moment of $f_\alpha(n)$ on looking at Eq. (\ref{eq18-2}).

It is clear that
\begin{align*}
\sum_{n\le x}f_\alpha^2(n)\le \sum_{\substack{p_1+b_1=p_2+b_2\le x\\ b_1,b_2<(\log x)^{1/\alpha}}}1.
\end{align*}
By separating the cases that $b_1=b_2$ or not, one can observe easily that
\begin{align}\label{equation-18-1}
\sum_{n\le x}f_\alpha^2(n)\le \sum_{b_1<(\log x)^{1/\alpha}}\sum_{p_1\le x}1+2\sum_{b_1<b_2<(\log x)^{1/\alpha}}\sum_{\substack{p_2<p_1\le x\\ p_1-p_2=b_2-b_1}}1.
\end{align}
Using again the prime number theorem, we get
\begin{align}\label{equation-18-2}
\sum_{b_1<(\log x)^{1/\alpha}}\sum_{p_1\le x}1\ll \frac{x}{\log x}\mathcal{B}\left((\log x)^{1/\alpha}\right).
\end{align}
For even $h\neq0$, let $\pi_2(x,h)$ be the number of prime pairs $p$ and $q$ with $q-p=h$ not exceeding $x$. It is well known (see e.g. \cite[Theorem 7.3]{Na}) that
\begin{align}\label{c2}
\pi_2(x,h)\ll\frac{x}{\log^2x}\prod_{p|h}\left(1+\frac{1}{p}\right).
\end{align}
It is also well--known (see e.g. \cite{Halberstam}) that
\begin{align}\label{d2}
\prod_{p|h}\left(1+\frac{1}{p}\right)\ll\log\log h.
\end{align}
From Eqs. (\ref{c2}) and (\ref{d2}) we have
\begin{align}\label{equation-18-3}
\sum_{b_1<b_2<(\log x)^{1/\alpha}}\sum_{\substack{p_2<p_1\le x\\ p_1-p_2=b_2-b_1}}1&\ll \frac{x}{\log^2x}\sum_{b_1<b_2<(\log x)^{1/\alpha}}\prod_{p|b_2-b_1}\left(1+\frac{1}{p}\right)\nonumber\\
&\ll_\alpha \frac{x\log\log\log x}{\log^2x}\mathcal{B}\left((\log x)^{1/\alpha}\right)^2.
\end{align}
Taking Eqs. (\ref{equation-18-2}) and (\ref{equation-18-3}) into Eq. (\ref{equation-18-1}), we conclude that
\begin{align}\label{equation-18-4}
\sum_{n\le x}f_\alpha^2(n)&\ll \frac{x}{\log x}\mathcal{B}\left((\log x)^{1/\alpha}\right)+\frac{x\log\log\log x}{\log^2x}\mathcal{B}\left((\log x)^{1/\alpha}\right)^2\nonumber\\
&\ll \frac{x\log\log\log x}{\log^2x}\mathcal{B}\left((\log x)^{1/\alpha}\right)^2
\end{align}
since $\mathcal{B}\left((\log x)^{1/\alpha}\right)\gg \left((\log x)^{1/\alpha}\right)^\alpha=\log x$. Now, it follows immediately from Eqs. (\ref{eq18-2}), (\ref{eq18-3}) and (\ref{equation-18-4}) that
$$
\mathcal{S}_\alpha(x)\gg_{c,\alpha} x/\log\log\log x,
$$
which completes the proof of our theorem.
\end{proof}

\section{Proof of Theorem \ref{th2}}

\begin{proof}[Proof of Theorem \ref{th2}]
Let $p_i$ be the $i$-th prime and $d_t=p_1p_2\cdot\cdot\cdot p_t$ for any positive integer $t$.
For any positive integer $j$, set
$$\mathcal{B}_j=\left\{n:n\in\mathbb{N},d_{j}|n\right\}\cap \left[\exp\left(j^m\right),\exp\left((j+1)^m\right)\right)$$
and
$$\mathcal{B}=\bigcup_{j=1}^{\infty}\mathcal{B}_j.$$
It is not difficult to see that there is some positive integer $j_0$ such that
\begin{align}\label{new19-1}
|\mathcal{B}_{j_0}|<|\mathcal{B}_{j_0+1}|<|\mathcal{B}_{j_0+2}|<|\mathcal{B}_{j_0+3}|<\cdot\cdot\cdot
\end{align}
via the prime number theorem.
Let $x$ be a sufficiently large number. Suppose that
\begin{align}\label{eq0}
\exp\left(\ell^m\right)\le x<\exp\left((\ell+1)^m\right),
\end{align}
then
\begin{align}\label{eq1}
(\log x)^{1/m}-1<\ell\le (\log x)^{1/m}.
\end{align}
By Eq. (\ref{eq1}) together with the Chebyshev estimate, we have
\begin{align}\label{eq2}
d_{\ell}&=\exp\left(\sum_{p\le p_{\ell}}\log p\right)
<\exp\left(2\ell\log \ell\right)
\le\exp\left(\frac{2}{m}(\log x)^{1/m}\log\log x\right),
\end{align}
\begin{align}\label{new19-2}
d_{\ell-1}&=\exp\Bigg(\sum_{p\le p_{\ell-1}}\log p\Bigg)
>\exp\left(\frac{2}{3}\ell\log \ell\right)
\ge\exp\left(\frac{1}{2m}(\log x)^{1/m}\log\log x\right),
\end{align}
and
\begin{align}\label{eq3}
\exp\left((\ell-1)^m\right)<\exp\left(\ell^m-(m-1/2)\ell^{m-1}\right)
<x\exp\left(-(m-1)(\log x)^{\frac{m-1}{m}}\right).
\end{align}
By Eqs. (\ref{eq2}), (\ref{eq3}) and definition of $\mathcal{B}$, we have
\begin{align*}
\mathcal{B}(x)&
\ge\mathcal{B}_\ell(x)+\left|\mathcal{B}_{\ell-1}\right|\nonumber\\&\ge\frac{x-\exp\left(\ell^m\right)}{d_\ell}+\frac{\exp\left(\ell^m\right)-\exp\left((\ell-1)^m\right)}{d_{\ell-1}}-3\nonumber\\
&\ge \frac{x-\exp\left((\ell-1)^m\right)}{d_\ell}-3\nonumber\\
&\gg x\exp\left(-\frac{2}{m}(\log x)^{1/m}\log\log x\right).
\end{align*}
The subset $\mathcal{B}$ is the one satisfying $c$-condition, since the elements of $\mathcal{B}$ lying in the interval $[x/2,x)$ are divisible by more (or at least equal) primes than those in the interval $[1,x/2)$, which leads to the fact that there are more elements of $\mathcal{B}$ located in $[1,x/2)$ than in $[x/2,x)$. Therefore, we have $\mathcal{B}(x/2)\ge \mathcal{B}(x)/2$. It remains to prove $$\mathcal{S}(x)\ll x/\log\log\log x.$$ To this aim, let $s$ be a sufficiently large integer depending on $x$ and $m$ which shall be decided later.
The sums $p+b~(b\in \mathcal{B})$ up to $x$ will divided into the following three parts:

Part I are the sums $p+b$ with $p>p_\ell$ and $b\in \mathcal{B}_j~(s\le j\le \ell)$. The sums $p+b$ in this part cannot be divisible by primes $p_j~(1\leqslant j\leqslant s)$, thus the number of the sums $p+b$ within this part is not large than
\begin{align}\label{eq5}
\sum_{\substack{n\le x\\\left(n,\prod_{p\le p_s}p\right)=1}}1
&=\sum_{k|\prod_{p\le p_s}p}\mu(k)\left\lfloor \frac{x}{k}\right\rfloor\nonumber\\ &\le x\prod_{p\le p_s}\left(1-\frac{1}{p}\right)+O\left(2^s\right)\nonumber\\
&\ll x(\log p_s)^{-1}+2^{s}\nonumber\\
&\ll x(\log s)^{-1}+2^s,
\end{align}
where the last but one step follows from the Mertens estimate.

Part II are the sums $p+b$ with $p\le p_\ell$ and $b\in \mathcal{B}_j~(s\le j\le \ell).$ The number of these sums $p+b$ up to $x$ can be bounded by Eqs. (\ref{new19-1}), (\ref{eq1}), (\ref{new19-2}) and the Chebyshev estimate as
\begin{align}\label{eq6}
\left(|\mathcal{B}_s|+\cdot\cdot\cdot+|\mathcal{B}_{\ell-1}|+\mathcal{B}_\ell(x)\right)\ell&\le \left(\ell|\mathcal{B}_{\ell-1}|+\mathcal{B}_\ell(x)\right)\ell\nonumber\\
&\le \left((\log x)^{1/m}\frac{\exp\left(\ell^m\right)}{d_{\ell-1}}+\frac{x}{d_\ell}\right)(\log x)^{1/m}\nonumber\\
&\le 2x(\log x)^{2/m} /d_{\ell-1}\nonumber\\
&\ll x(\log x)^{2/m}\exp\left(\!-\frac{1}{2m}(\log x)^{1/m}\log\log x\! \right).
\end{align}

Part III are the remaining sums $p+b$ with $b\in \mathcal{B}_j~(1\le j\le s)$. These remaining sums are trivially bounded by
\begin{align}\label{eq7}
s|\mathcal{B}_s|\pi(x)\ll s\frac{\exp\left((s+1)^m\right)}{d_s}\frac{x}{\log x}\ll\frac{x}{\log x}\exp\left((s+1)^m\right).
\end{align}

Gathering together Eqs. (\ref{eq5}), (\ref{eq6}) and (\ref{eq7}), we would have
\begin{align}\label{eq8}
\mathcal{S}(x)\ll &x(\log x)^{2/m}\exp\left(-\frac{1}{2m}(\log x)^{1/m}\log\log x \right)\nonumber\\&\quad\quad\quad\quad\quad\quad+2^s+x(\log s)^{-1}+\frac{x}{\log x}\exp\left((s+1)^m\right).
\end{align}
Now we choose $s$ to be the integer $\left\lfloor (\log\log x)^{\frac{1}{2m}}\right\rfloor-1$, then Eq. (\ref{eq8}) yields
\begin{align}
\mathcal{S}(x)\ll x(\log\log\log x)^{-1}.
\end{align}
This completes the proof of Theorem \ref{th2}.
\end{proof}

\section*{Acknowledgments}
The author is supported by National Natural Science Foundation of China  (Grant No. 12201544), Natural Science Foundation of Jiangsu Province, China (Grant No. BK20210784), China Postdoctoral Science Foundation (Grant No. 2022M710121). He is also supported by  (Grant No. JSSCBS20211023) and (Grant No. YZLYJF2020PHD051).


\begin{thebibliography}{KMP}
\bibitem{Ch1} Y.--G. Chen, {\it On integers of the forms $k^r-2^n$ and $k^r2^n+1$,} J. Number Theory {\bf 98} (2003), 310--319.
\bibitem{Ch2} Y.--G. Chen, {\it On integers of the forms $k\pm2^n$ and $k2^n\pm1$,} J. Number Theory {\bf 125} (2007), 14--25.
\bibitem{Ch3} Y.--G. Chen, {\it Romanoff theorem in a sparse set,} Sci. China Math. {\bf53} (2010), 2195--2202.
\bibitem{va} J. G. van der Corput, On de Polignac's conjecture, Simon Stevin {\bf27} (1950), 99--105.
\bibitem{Ding} Y. Ding, {\it On a problem of Romanoff type,} Acta Arith. {\bf205} (2022), 53--62.
\bibitem{ES} C. Elsholtz, J.--C. Schlage--Puchta, {\it On Romanov's constant,} Math. Z. {\bf288} (2018), 713--724.
\bibitem{Els} C. Elsholtz, F. Luca, S. Planitzer, {\it Romanov type problems,} Ramanujan J. {\bf47} (2018), 267--289.
\bibitem{Er} P. Erd\H{o}s, {\it On the integers of the form $2^k+p$ and some related problems,} Summa Brasil. Math. {\bf 2} (1950), 113--123.
\bibitem{Halberstam} H. Halberstam and H. Richert, {\it Sieve Methods,} Academic Press. London (1974).
\bibitem{Na} M. B. Nathanson, {\it Additive Number Theory: The Classical Bases}, Springer, 1996.
\bibitem{Pa3} H. Pan, H. Li, {\it The Romanoff theorem revisited,} Acta Arith. {\bf 135} (2008), 137--142.
\bibitem{de1} A. de Polignac, {\it Six propositions arithmologiques d\'{e}duites du crible d'Eratosth\'{e}ne,} Nouv. Ann. Math. {\bf8} (1849), 423--429.
\bibitem{de2} A. de Polignac, {\it Recherches nouvelles sur ies nombres primiers,} C. R. Acad. Sci. Paris {\bf 29} (1849), 738--739.
\bibitem{Ro} N. P. Romanoff, {\it\"{U}ber einige S\"{a}tze der additiven Zahlentheorie,} Math. Ann. {\bf 109} (1934), 668--678.
\bibitem{Ch} Q.--H. Yang, Y.--G. Chen, {\it On the integers of the form $p+b$,} Taiwanese J. Math. {\bf 18} (2014), 1623--1631.
\end{thebibliography}
\end{document}